\documentclass[12pt,amscd]{amsart}
\usepackage[all]{xy}
\usepackage{graphicx,tikz-cd}
\usepackage{amsmath,amsxtra,amssymb,latexsym, amscd,amsthm}
\usepackage{indentfirst}
\usepackage[mathscr]{eucal}
\usepackage[pagebackref=true]{hyperref}

\newtheorem{thm}{Theorem}[section]
\newtheorem{cor}[thm]{Corollary}
\newtheorem{lem}[thm]{Lemma}
\newtheorem{prop}[thm]{Proposition}
\theoremstyle{definition}

\newtheorem{rem}[thm]{Remark}

\numberwithin{equation}{section}

\DeclareMathOperator{\NN}{\mathbb {N}}
\DeclareMathOperator{\ZZ}{\mathbb {Z}}

\DeclareMathOperator{\htt}{ht}

\DeclareMathOperator{\un}{un}

\def\a {\mathbf a}

\def\k {\mathrm{k}}

\begin{document}

\title{Multiplicity of powers of squarefree monomial ideals}

\author{Phan Thi Thuy}
\address{Department of Mathematics, Hanoi National University of Education, 136 Xuan
Thuy, Hanoi, Vietnam}
\email{thuysp1@gmail.com}

\author{Thanh Vu}
\address{Institute of Mathematics, VAST, 18 Hoang Quoc Viet, Hanoi, Vietnam}
\email{vuqthanh@gmail.com}


\date{}

\commby{}

\begin{abstract}
    Let $I$ be an arbitrary nonzero squarefree monomial ideal of dimension $d$ in a polynomial ring $S = \k[x_1,\ldots,x_n]$. Let $\mu$ be the number of associated primes of $S/I$ of dimension $d$. We prove that the multiplicity of powers of $I$ is given by 
    $$e_0(S/I^s) = \mu \binom{n-d+s-1}{s-1},$$ 
    for all $s \ge 1$. Consequently, we compute the multiplicity of all powers of path ideals of cycles.
\end{abstract}

\maketitle

\section{Introduction}
\label{sect_intro}
Let $S = \k[x_1,\ldots,x_n]$ be a standard graded polynomial ring over a field $\k$ and $M$ a finitely generated graded $S$-module. The Hilbert function of $M$ is given by $H_M(a) = \dim_\k M_a$ for all $a \in \ZZ$. Hilbert proved that when $a$ is large enough, $H_M$ agrees with a polynomial of degree $d = \dim (M)-1$ in $a$, called the Hilbert polynomial of $M$. Furthermore, the Hilbert polynomial of $M$ can be written as 
$$P_M(z) = \sum_{i=0}^d (-1)^i e_i(M) \binom{z + d -i}{d-i}$$
where $e_i(M)$ are integers. The first coefficient $e_0(M)$ is called the multiplicity of $M$. When $I$ is a homogeneous ideal of $S$, Herzog, Puthenpurakaland, and Verma \cite[Theorem 1.1]{HPV} proved that $e_i(M/I^kM)$ is of polynomial type of degree $\le n - d + i$ for all $i = 0, \ldots, d$, where $d = \dim (M/IM)$. Recently, Shan, Wang, and Lu \cite{SWL} computed the multiplicity of powers of path ideals of paths. To our knowledge and surprise, there are no other known examples of a nontrivial family of ideals for which the multiplicity of powers is explicitly written down. On the other hand, for arbitrary squarefree monomial ideal $I$, by \cite[Corollary 6.2.3]{HH}, $e_0(S/I)$ equals the number of associated primes $Q$ of $S/I$ of dimension $\dim(S/I)$. To understand this myth, we start by calculating the multiplicity of powers of path ideals of cycles. We realize a similar result to \cite{SWL} holds for arbitrary squarefree monomial ideals. In other words, for an arbitrary squarefree monomial ideal $I$, the multiplicity of powers of $I$ is already encoded in the dimension and the multiplicity of $I$ itself.

\begin{thm}\label{thm_mul} Let $I$ be a nonzero squarefree monomial ideal of dimension $d$ in $S$. Let $\mu$ be the number of associated primes of $S/I$ of dimension $d$. Then 
    $$e_0(S/I^s) = \mu \binom{n-d+s-1}{s-1},$$ 
    for all $s \ge 1$.
\end{thm}

We then compute the multiplicity of $d$-path ideals of $n$-cycles. 

\begin{thm}\label{thm_path_ideals_cycles} Let $n > d \ge 2$ be positive integers and 
$$I_{n,d} = (x_1\cdots x_d,x_{2}\cdots x_{d+1}, \cdots, x_n x_1 \cdots x_{d-1}) \subset S = \k[x_1,\ldots,x_n]$$
be the $d$-path ideal of an $n$-cycle. Write $n = kd + r$ for some $k \ge 1$ and $1 \le r \le d$. Then 
$$e_0(S/I_{n,d}) = d \binom{k+d -r}{k} - \binom{k+d -r}{k+1}.$$
\end{thm}

\section{Multiplicity of ideals}\label{sec_pre}

Throughout the paper, let $S = \k[x_1, \ldots, x_n]$. The materials in this section are well-known to the experts in algebraic geometry. Nonetheless, we could not find explicit references, we include an argument here for completeness. We refer to \cite[Section I.7]{Ha} for more details.

\begin{lem}\label{lem_unmixed} Let $I$ be a nonzero homogeneous ideal of $S$. Assume that $I$ has a primary decomposition $I = Q_1 \cap \cdots \cap Q_s \cap Q_{s+1} \cap \cdots \cap Q_t$ where $\htt(Q_i) = \htt(I)$ for $i = 1, \ldots, s$ and $\htt(Q_j) > \htt(I)$ for $i = s+1, \ldots, t$. Denote the unmixed part of $I$ by by $I^{\un} = Q_1 \cap \cdots \cap Q_s$. Then 
$$e_0(S/I) = e_0(S/I^{\un}) = \sum_{i=1}^s e_0(S/Q_i).$$    
\end{lem}
\begin{proof} Let $M$ be a finitely generated graded $S$-module. Denote by $d = \dim (M)$. By Hilbert's Theorem \cite[Theorem 6.1.3]{HH}, there exists a Laurent polynomial $Q_M(t) \in \ZZ[t,t^{-1}]$ with $Q_M(1) > 0$ such that the Hilbert's series of $M$ is
$$\operatorname{Hilb}_M(t) = \frac{Q_M(t)}{(1-t)^d}.$$
In particular, $e_0(M) = Q_M(1)$ and $e_0(M)/(d-1)!$ is the leading coefficient of the Hilbert polynomial of $M$. 

First, assume that $t > s$. Let $L = Q_{s+1} \cap \cdots \cap Q_t$. From the exact sequence 
$$0 \to S/I \to S/I^{\un} \oplus S/L \to S/(I^{\un} + L) \to 0$$
and the additivity of the Hilbert polynomial, we deduce that 
$$P_{S/I^{\un}}(t) + P_{S/L}(t) = P_{S/I}(t) + P_{S/(I^{\un} + L)}(t).$$
Since $\dim (S/L) < \dim(S/I)$ and $\dim(S/(I^{\un} + L)) < \dim(S/I)$, we deduce that $e_0(S/I) = e_0(S/I^{\un})$. 

We now assume that $I$ is unmixed and prove that $e_0(S/I) = \sum_{i=1}^s e_0(S/Q_i)$ by induction on $s$. The base case $s=1$ is vacuous. Assume that the statement holds for $s$. Let $J = Q_1 \cap \cdots \cap  Q_s$. With the same argument as above, it suffices to prove that 
\begin{equation}\label{eq_dim}
    \dim (S/(J + Q_{s+1})) < \dim(S/Q_{s+1}).
\end{equation}
Let $P$ be any associated prime of $S/(J + Q_{s+1})$. Then we have $Q_{s+1} \subseteq P$. Hence, $Q = \sqrt{Q_{s+1}} \subseteq P$. If $Q = P$ then $J \subseteq Q$, which is a contradiction. Hence, $Q \subsetneq P$. The conclusion follows.    
\end{proof}
\begin{rem} Assume that $I$ is a radical ideal of $S$. Then $I^{(s)} = P_1^s \cap \cdots \cap P_m^s$, where $I^{(s)}$ is the $s$-th symbolic power of $I$ and $P_1, \ldots, P_m$ are the minimal associated primes of $S/I$. By Lemma \ref{lem_unmixed}, we have    
$$e_0\left ( S/I^{(s)} \right ) =  e_0(S/I^s).$$
\end{rem}
\begin{proof}[Proof of Theorem \ref{thm_mul}] The conclusion follows from Lemma \ref{lem_unmixed} and the following facts
\begin{enumerate}
    \item associated primes of monomial ideals are generated by variables of $S$;
    \item $e_0(S/(x_1,\ldots,x_{n-d})^s) = \binom{n-d+s-1}{s-1}$.
\end{enumerate}
The conclusion follows.
\end{proof}

\section{Multiplicity of path ideals of cycles}
In this section, we extend the work of Shan, Wang, and Lu \cite{SWL} to the case of path ideals of cycles. Let $n \ge d \ge 2$ be positive integers and $S = \k[x_1,\ldots, x_n]$ be a standard graded polynomial ring over a field $\k$. We denote by 
$$g_1 = x_1\cdots x_d, g_2 = x_{2}\cdots x_{d+1}, \cdots, g_n = x_n x_1 \cdots x_{d-1},$$
and $I_{n,d} = (g_1, \ldots, g_n)$ the $d$-path ideal of an $n$-cycle. In general, path ideals of cycles are more subtle than path ideals of paths; see \cite{AF, BCV1, BCV2} for more information. First, we describe the associated primes of $S/I_{n,d}$.

\begin{prop}\label{prop_ass_cyc} Let $1 \le a_1 < a_2 < \cdots < a_s \le n$ be positive integers. Denote by $a_{s+1} = a_1 + n$ and $a_{s+2} = a_2 + n$. Then $P_\a = (x_{a_1}, \ldots, x_{a_s})$ is an associated prime of $S/I_{n,d}$ if and only if the following conditions hold:
\begin{enumerate}
    \item $a_{i+1} - a_i \le d$,
    \item $a_{i+2} - a_i > d$, 
\end{enumerate}
  for $i = 1, \ldots, s$.   
\end{prop}
\begin{proof} First, we prove the sufficient condition. We assume that $\a = (a_1, \ldots, a_s)$ satisfies the two conditions above. Let $U = \{1,\ldots,n\} \setminus \{a_1, \ldots, a_s\}$ and $f = \prod_{i \in U} x_i$. Condition (1) implies $f \notin I_{n,d}$. Since $I_{n,d}$ is squarefree, we deduce that $x_j \notin I_{n,d} : f$ for all $j \in U$. Conditions (1) and (2) imply that 
$$x_{a_i} = g_{a_{i-1}+1} / \gcd(g_{a_{i-1}+1},f) \text{ for all } i =1, \ldots, s.$$ 
Hence, $I_{n,d} : f = P_{\a}$. Therefore, $P_\a$ is an associated prime of $S/I_{n,d}$. 

Conversely, assume that $P_\a$ is an associated prime of $S/I_{n,d}$. By \cite[Corollary 1.3.10]{HH}, there exists a monomial $f$ of $S$ such that $P_\a = I_{n,d} : f$. Since $I_{n,d}$ is squarefree, we may choose $f$ to be squarefree. In particular, $f \notin I_{n,d}$. Assume by contradiction that $a_{i+1} - a_i > d$ for some $i$. Then $g = x_{a_i+1} \cdots x_{a_i + d-1} \in I_{n,d} \subset P_\a$, which is a contradiction. Hence, $a_{i+1} - a_i \le d$. Now, assume by contradiction that $a_{i+2} - a_i \le d$ for some $i$. We have $x_{a_i}, x_{a_{i+2}}$ does not divide $f$. Now for any minimal generator $g$ of $I_{n,d}$ such that $x_{i+1} | g$ we must have either $x_{a_i}$ or $x_{a_{i+2}}$ divides $g$. In other words, $g / \gcd (g,f)$ is divisible by $x_{a_i} x_{a_{i+1}}$ or $x_{a_{i+1}} x_{a_{i+2}}$. Hence, $x_{a_{i+1}} \notin I_{n,d} : f$, a contradiction. The conclusion follows.   
\end{proof}
Consequently, we have
\begin{cor}\label{cor_dim_cyc} Let $n \ge d \ge 2$ be positive integers. Then $\dim (S/I_{n,d}) = n - \lceil \frac{n}{d} \rceil$. 
\end{cor}
\begin{proof} We use the notation as in Proposition \ref{prop_ass_cyc}. Let $1 \le a_1 < a_2 < \cdots < a_s \le n$ be positive integers such that $P_\a$ is an associated prime of $S/I_{n,d}$. By Proposition \ref{prop_ass_cyc}, we have $a_{i+1} - a_i \le d$ for all $i = 1, \ldots, s$. In particular, $n = a_{s+1} - a_1 \le s d$. Hence, $s \ge \lceil \frac{n}{d} \rceil$ and $\dim (S/I_{n,d}) \le n - \lceil \frac{n}{d} \rceil$.

Furthermore, let $s = \lceil \frac{n}{d} \rceil$ and $a_i = 1 + (i-1)d$ for all $i = 1, \ldots, s$. Then $\a = (a_1, \ldots, a_s)$ satisfies the condition of Proposition \ref{prop_ass_cyc}. Hence, $P_\a$ is an associated prime of $S/I_{n,d}$. The conclusion follows.
\end{proof}

We now fix the following notation throughout the rest of the paper. 
\begin{enumerate}
    \item $n = kd + r$ for some integers $k,r$ such that $k \ge 1$ and $1 \le r \le d$;
    \item $\a = (a_1, \ldots, a_{k+1}) \in \NN^{k+1}$ be such that $1 \le a_1 < a_2 < \cdots < a_{k+1} \le n$;
    \item $a_{k+2} = a_1 + n$ and $b_j = a_{j+1} - a_j$ for all $j = 1, \ldots, k+1$.
\end{enumerate}
To compute the multiplicity of $S/I_{n,d}$ we need some preparation lemmas.
\begin{lem}\label{lem_red_1} With the notations as above we have $P_\a$ is an associated prime of $S/I_{n,d}$ if and only if the following conditions hold
\begin{enumerate}
    \item $b_i \le d$ for all $i = 1, \ldots k$,
    \item $\sum_{i=1}^k b_i \ge (k-1)d + r$,
    \item $a_1 + \sum_{i=1}^k b_i \le kd +r$.
\end{enumerate}    
\end{lem}
\begin{proof} By definition and Proposition \ref{prop_ass_cyc}, it suffices to prove that the conditions above implies that $b_i + b_{i+1} > d$ for all $i = 1, \ldots, k+1$. Indeed, if $b_i + b_{i+1} \le d$ for some $i \le k-1$, then $\sum_{i=1}^k b_i \le (k-2)d + d < (k-1)d + r$. This is a contradiction, the conclusion follows.    
\end{proof}
\begin{lem}\label{lem_red_2} For each positive integers $k,d$, let 
$$ U_{k,d}  = \left \{ (c_1, \ldots, c_{k+1}) \in \NN^{k+1} \mid c_i \le d, \sum_1^{k+1} c_i = kd \right \}.$$
Then $|U_{k,d} | = \binom{d+k}{d}$.
\end{lem}
\begin{proof}
    Each $(c_1, \ldots, c_{k+1}) \in U_{k,d}$ gives a monomial $x_1^{c_1} \cdots x_{k+1}^{c_{k+1}}$ of degree $kd$ in 
    $$M = \k[x_1,\ldots,x_{k+1}]/(x_1^{d+1},\ldots,x_{k+1}^{d+1}).$$
    Hence, 
    $$|U_{k,d}| = H_M(kd) = \binom{d+k}{d}.$$
    The conclusion follows.
\end{proof}

\begin{proof}[Proof of Theorem \ref{thm_path_ideals_cycles}] Let $n = kd + r$ for some integers $k,r$ such that $k \ge 1$ and $1 \le r \le d$. By Corollary \ref{cor_dim_cyc}, $\dim(S/I_{n,d}) = n - (k+1)$. By Lemma \ref{lem_unmixed} and Lemma \ref{lem_red_1}, we deduce that $e_0(S/I_{n,d})$ is equal to the cardinality of the set $V_{k,d,r}$ consisting of tuples $(a_1, b_1, \ldots, b_k) $ satisfying conditions of Lemma \ref{lem_red_1}. For each integer $s$, let 
$$W_{k,d,r,s} = \{ (b_1, \ldots, b_k) \in \NN^k \mid b_i \le d \text{ for all } i = 1, \ldots, k, \text{ and } \sum_{i=1}^k b_i = (k-1) d + r + s\}.$$
For each $(b_1, \ldots, b_k) \in W_{k,d,r,s}$ with $0 \le s \le d-r$ there are $d-s$ choices of $a_1$ so that $(a_1, b_1, \ldots, b_k) \in V_{k,d,r}$. Hence, 
$$|V_{k,d,r}| = \sum_{s= 0}^{d-r} (d-s) | W_{k,d,r,s}|.$$
Now, each $(b_1, \ldots, b_k) \in W_{k,d,r,s}$ must have $b_i \ge r + s$, as $b_i \le d$ for all $i = 1, \ldots, k$. Let $c_i = b_i - r - s$ for $i = 1, \ldots, k$. We then deduce that $W_{k,d,r,s} \cong U_{k-1,d-r-s}$. By Lemma \ref{lem_red_2}, we deduce that 
\begin{align*}
    e_0(S/I_{n,d}) &= |V_{k,d,r}| = \sum_{s=0}^{d-r} (d-s) \binom{d+k-r-s-1}{k-1}\\
    &=\sum_{j=r}^d j \binom{k-1 -r +j}{k-1}, \text{ setting } j = d-s \\
    & = r \sum_{\ell=0}^{d-r} \binom{k-1 + \ell }{k-1} + \sum_{\ell=0}^{d-r} \ell \binom{k-1+\ell}{k-1}, \text{ setting } \ell = j - r\\
    & = r \binom{k + d -r}{k} + (d-r) \binom{k+d-r}{k} - \binom{k+d-r}{k+1},
\end{align*}
where the last equality follows from standard binomial identities, see e.g. \cite{G}. The conclusion follows.
\end{proof}

\begin{cor} Let $n = kd + r$ where $k,d,r$ are positive integers such that $1 \le r \le d$. Then 
$$e_0(S/I_{n,d}^s) = \left ( d \binom {k+d-r}{k} - \binom{k+d-r}{k+1} \right) \binom{ k+s}{s-1}.$$  
\end{cor}
\begin{proof}
    The conclusion follows from Theorem \ref{thm_mul}, Corollary \ref{cor_dim_cyc}, and Theorem \ref{thm_path_ideals_cycles}.
\end{proof}


\end{document}